\documentclass{birkjour}[10pt]
 \newtheorem{thm}{Theorem}[section]
 
 \newtheorem{lem}[thm]{Lemma}
 
 \theoremstyle{definition}
 
 \theoremstyle{remark}
 \newtheorem{rem}[thm]{Remark}
 \newtheorem{ex}{Example}
 \numberwithin{equation}{section}
\usepackage{xcolor}
\usepackage{multirow}
\usepackage{makecell}
\usepackage{cellspace}
\usepackage{matlab-prettifier}
\usepackage{graphicx, epstopdf}
\usepackage{float}
\allowdisplaybreaks[1]
\begin{document}

%
%
%
%
%
%
%
%
%

\title[Random average sampling]
 {Random average sampling and reconstruction in shift-invariant subspaces of mixed Lebesgue spaces}

\author[Ankush Kumar Garg]{Ankush Kumar Garg}
\address{%
School of Mathematics,   \\
Indian Institute of Science Education\\
and Research Thiruvananthapuram,\\
Maruthamala P.O., Vithura,\\
 Thiruvananthapuram-695551.}

\email{ankush16@iisertvm.ac.in}

\author{S. Arati}
\address{School of Mathematics,   \\
Indian Institute of Science Education\\
and Research Thiruvananthapuram,\\
Maruthamala P.O., Vithura,\\
Thiruvananthapuram-695551.}
\email{aratishashi@iisertvm.ac.in}

\author{P. Devaraj}
\address{School of Mathematics,   \\
Indian Institute of Science Education\\
and Research Thiruvananthapuram,\\
Maruthamala P.O., Vithura,\\
Thiruvananthapuram-695551.}
\email{devarajp@iisertvm.ac.in}

\date{\today}

\maketitle
\begin{abstract}
 In this paper, the problem of reconstruction of signals in mixed Lebesgue spaces from their random average samples has been studied. Probabilistic sampling inequalities for certain subsets of shift-invariant spaces have been derived. It is shown that the probabilities increase to one when the sample size increases. Further, explicit reconstruction formulae for signals in these subsets have been obtained for which the numerical simulations have also been performed. 
\end{abstract}

\section{Introduction}

In information theory, it is highly desirable to gain maximum  information about a function (or signal) using the least available data. Image processing, data analysis, computer tomography, bio-engineering and artificial intelligence are a few fields which often deal with sampling and reconstruction problems. In 1949, the celebrated Shannon sampling theorem was proved which turned out to be a milestone in this field of study and set the foundation for information theory. Over these years, the theory of sampling has been intensively studied. For a detailed survey on the theory, we refer to Butzer and Stens \cite{ButzerStens}. Further, the theory has been developed for shift-invariant spaces with a single generator as well as with multiple generators in various contexts such as regular, irregular, average and multi-channel sampling by several authors.\\

 The average sampling is a useful and a practical model of sampling when the physical devices fail to measure the exact value of a signal at a given time. Sun and Zhou gave the reconstruction formulae from local convoluted samples for band-limited functions in \cite{sun5, sun6}, for spline subspaces in \cite{sun3, sun2} and for shift-invariant subspaces with symmetric averaging functions in \cite{sun4}. The average sampling and reconstruction algorithms for shift-invariant subspaces were also studied in \cite{AldroubiSunTang, kan, deva1}. In \cite{NashedSunXian}, the average sampling has been analyzed for a reproducing kernel subspace of $L^{p}$- spaces. Li et al. gave the reconstruction formula for the functions in local shift-invariant subspaces of $L^{2}(\mathbb{R}^{d})$ from average random samples in \cite{LiWenXian2}. \\

Random sampling is another type of sampling which has been used in practical applications such as learning theory, image filtering and compressed sensing. In \cite{bass-grochenig1}, Bass and Gröchenig discussed random sampling for multivariate trigonometric polynomials. They obtained the probabilistic sampling inequality for band-limited functions on $\mathbb{R}$ in \cite{bass-grochenig2} and the same for band-limited functions on $\mathbb{R}^{d}$ in \cite{bass-grochenig3}. Random sampling in shift-invariant spaces was studied in \cite{yangwei, yang, furr}. Yang and Tao in \cite{YangTao} studied random sampling and gave an approximation model for signals having bounded derivatives. Random sampling for reproducing kernel subspaces was analyzed in \cite{patelsampath, jiangli2, LiSunXian}.  
  The random sampling theory was also extended to mixed Lebesgue spaces which are generalized versions of the Lebesgue spaces. In this context, the random sampling for reproducing kernel subspaces was studied in \cite{Goyalpatelsampath} and that for multiply generated shift-invariant subspaces was analyzed in \cite{jiangli3}. For a detailed study of the mixed Lebesgue spaces, we refer to \cite{benzone}.\\
  
  This paper deals with the study of random average sampling for functions in mixed Lebesgue spaces using a probabilistic approach.
For $1 < p,q < \infty$, let $L^{p,q}(\mathbb{R} \times \mathbb{R}^{d})$ denote the mixed Lebesgue space, which consists of complex valued measurable functions on $\mathbb{R} \times \mathbb{R}^{d} $ such that
 $$ \bigg(\int_{\mathbb{R}}\bigg(\int_{\mathbb{R}^{d}}|f(x,y)|^{q}dy \bigg )^{p/q}dx \bigg)^{1/p} < \infty$$
  and let $L^{\infty, \infty}(\mathbb{R} \times \mathbb{R}^{d})$ denote the set of all complex valued measurable functions on $\mathbb{R}^{d+1}$ such that $\|f\|_{L^{\infty, \infty}(\mathbb{R}\times \mathbb{R}^{d})} := ess \sup |f| < \infty$.

\noindent
 Similarly, for $ 1<p,q < \infty$, $l^{p,q}(\mathbb{Z}\times \mathbb{Z}^{d})$ denotes the space of all complex sequences $c=\big(c(k_{1},k_{2})\big)_{(k_{1}\in \mathbb{Z}, k_{2}\in \mathbb{Z}^{d})}$ such that 
$$\|c\|_{l^{p,q}} := \left(\sum_{k_{1}\in \mathbb{Z}}\left(\sum_{k_{2} \in \mathbb{Z}^{d}}|c(k_{1}, k_{2})|^{q}\right)^{p/q}\right)^{1/p} < \infty$$ and $l^{\infty, \infty}(\mathbb{Z} \times \mathbb{Z}^{d})$ denotes the space of all complex sequences on $\mathbb{Z}^{d+1}$ such that $\|c\|_{l^{\infty, \infty}} := \displaystyle \sup_{k \in \mathbb{Z}^{d+1}}|c(k)| < \infty.$ We observe that $L^{p,p}(\mathbb{R} \times \mathbb{R}^{d}) = L^{p}(\mathbb{R}^{d+1})$ and $l^{p,p}(\mathbb{Z} \times \mathbb{Z}^{d}) = l^{p}(\mathbb{Z}^{d+1})$ for $1 < p < \infty.$

\bigskip
\noindent
We consider a multiply generated shift-invariant subspace of the mixed Lebesgue space $L^{p,q}(\mathbb{R} \times \mathbb{R}^{d}), 1 < p,q < \infty,$ of the form $$ V^{p,q}(\Phi)= \bigg \{\sum_{k_{1}\in \mathbb{Z}}\sum_{k_{2} \in \mathbb{Z}^{d}} {\textbf{c}^{T}(k_{1}, k_{2})}\Phi(\cdot -k_{1},\cdot - k_{2}): \textbf{c} \in \big(l^{p,q}(\mathbb{Z}\times\mathbb{Z}^{d})\big)^{r} \bigg\}$$ where $\Phi= (\phi_{1}, \phi_{2}, \dots, \phi_{r})^{T}$ with $\phi_{i} \in L^{p,q}(\mathbb{R} \times \mathbb{R}^{d})$ and $\textbf{c}= (c_{1}, c_{2}, \dots, c_{r})^{T}.$ \\

We prove the sampling inequalities for certain subsets of $V^{p,q}(\Phi)$ and estimate the probabilities with which they hold. Our results show that the probability tends to one as the number of samples increases. Further, using these sampling inequalities, we derive explicit reconstruction formulae. We also simulate the reconstruction numerically for certain examples and obtain the error estimates.

\section{Preliminaries} \label{sec2}
The subsets of $V^{p,q}(\Phi)$ that are needed for our analysis are defined as follows :\\
 Let $C_{K} $ denote the compact set $[-K_{1}, K_{1}]\times [-K_{2}, K_{2}]^{d} \subseteq \mathbb{R}\times \mathbb{R}^{d}$ for some positive constants $K_{1}$ and $K_{2}$ and $V^{p,q}(\Phi, \delta, C_{K})$ be the set 
 \begin{equation*}V^{p,q}(\Phi, \delta, C_{K}):= \big \{f \in V^{p,q}(\Phi): \|f\|_{L^{p,q}(C_K)} \geq(1-\delta)\|f\|_{L^{p,q}(\mathbb{R}\times \mathbb{R}^{d})} \big \},
  \end{equation*}
   where $0 < \delta <1$. In fact, $V^{p,q}(\Phi, \delta, C_{K})$ consists of the signals in $V^{p,q}(\Phi)$ whose energy is concentrated on the set $C_{K}.$

\noindent
For a positive integer $N$, $V_{N}^{p,q}(\Phi)$ denotes the finite dimensional subspace of  $V^{p,q}(\Phi)$ given by
\begin{equation*}
V_{N}^{p,q}(\Phi):= \bigg\{\sum_{i=1}^{r} \sum_{|k_{1}|\leq N} \sum_{|k_{2}|\leq N} c_{i}(k_{1}, k_{2}) \phi_{i}(\cdot - k_{1}, \cdot- k_{2}): c_{i} \in l^{p,q}\bigg([-N, N] \times [-N, N]^{d}\bigg)\bigg \},
\end{equation*} 
where for a multi-index $k=(k_{1}, k_{2}, \dots, k_{d}) \in \mathbb{Z}^{d}, |k| := \displaystyle \max_{1 \leq i \leq d}|k_{i}|.$\\
 The unit ball of the above space is 
\begin{equation*}
V_{N}^{p,q, *}(\Phi) := \big \{g \in V_{N}^{p,q}(\Phi) : \|g\|_{L^{p,q}(\mathbb{R}\times
 \mathbb{R}^{d})} = 1\big \}. 
\end{equation*}

\noindent
For $\omega > 0$ and $\psi \in L^{1}(\mathbb{R}^{d+1}),$ the subset $V_{N, \omega, \psi}^{p,q}(\Phi)$ is defined as
\begin{equation*}
V_{N, \omega, \psi}^{p,q}(\Phi) := \big \{f \in  V_{N}^{p,q}(\Phi) : \|f * \psi\|_{L^{p,q}(C_{K})} \geq \omega \big \} 
\end{equation*}
and its unit ball is denoted by $V_{N, \omega, \psi}^{p,q,*}(\Phi).$\\
Further, for $0 < \mu \leq 1,$ we define 
\begin{eqnarray}
V_{N, \psi}^{p,q}(\Phi, \mu, C_{K}):= \bigg \{f \in V_{N}^{p,q}(\Phi) &:& \mu \| \psi\|_{{L}^{1}(\mathbb{R}^{d+1})} \|f\|_{L^{p,q}(\mathbb{R} \times \mathbb{R}^{d})} \nonumber\\
&&  \quad \leq  \int_{C_{K}}|(f*\psi)(x,y)|dx dy \bigg \} \label{rte1}
\end{eqnarray}
\noindent
Also, for $0 < \delta <1,$ we define $ V_{\psi}^{p,q}(\Phi, \delta, C_{K})$   as
\begin{eqnarray*}
 V_{\psi}^{p,q}(\Phi, \delta, C_{K}) &:=& \big\{ f \in V^{p,q}(\Phi, \delta, C_{K}) : \|f* \psi\|_{L^{p,q}(C_{K})} \\
&&  \quad \quad \quad \geq  (1-\delta) \|\psi\|_{L^{1,1}(C_{K})} \|f\|_{L^{p,q}(\mathbb{R} \times \mathbb{R}^{d})} \big \}
\end{eqnarray*} 
and $V_{\psi}^{p,q,*}(\Phi, \delta, C_{K}) $ as its unit ball.\\

\noindent
We make the following assumptions for our study:

\noindent
$(A_{1})$ The generators $\phi_{1}, \phi_{2}, \dots , \phi_{r}$ have stable shifts, i.e., there exist positive constants $\alpha_{1}$ and $\alpha_{2}$ such that
\begin{equation}
\alpha_{1}\|\textbf{c}\|_{l^{p,q}}\leq \bigg\|\sum_{k_{1}\in \mathbb{Z}} \sum_{k_{2}\in \mathbb{Z}^{d}}\textbf{c}^{T}(k_{1}, k_{2}) \Phi(\cdot -k_{1}, \cdot -k_{2})\bigg \|_{L^{p,q}(\mathbb{R}\times \mathbb{R}^{d})} \leq \alpha_{2}\|\textbf{c}\|_{l^{p,q}},\label{2eq1}
 \end{equation} 
  for $\textbf{c} \in \left(l^{p,q} (\mathbb{Z} \times \mathbb{Z}^{d})\right)^{r}$ and $ \| \textbf{c} \|_{l^{p,q}} = \sum_{i=1}^{r} \| c_{i}\|_{l^{p,q}} < \infty.$ \\
 Also $\phi_{i}$'s are continuous with polynomial decay, i.e.,
  \begin{equation}
 |\phi_{i}(x,y)| \leq \dfrac{\tilde{c}}{(1+|x|)^{s_{1}}(1+|y|)^{s_{2}}}, (x,y) \in (\mathbb{R}\times \mathbb{R}^{d}) ,\label{2eq2}
  \end{equation}  
   where $\tilde{c}$ is a positive constant and $s_{1}, s_{2}$ are positive constants satisfying \\
$s_{1}, s_{2} > d+1 - \frac{1}{p} -\frac{d}{q}$. \\~\\
$(A_{2})$ The averaging function $\psi \in L^{1}( \mathbb{R}^{d+1})$ with supp$(\psi) \subseteq C_{K}.$  \\~\\
$(A_{3})$ There exists a probability density function  $\rho $ defined over $C_{K}$ such that
 \begin{equation} \mathcal{C}_{\rho, 1} \leq \rho(x, y) \leq \mathcal{C}_{\rho, 2} \label{2eq3}
 \end{equation}
  for all $(x,y) \in C_{K},$ where $\mathcal{C}_{\rho, 1}$ and $\mathcal{C}_{\rho, 2}$ are positive constants.\\~\\
  We provide some examples of generators satisfying the assumption $(A_{1})$ in Section \ref{sec5}.

\section{Random average sampling inequalities} \label{sec3}
 In this section, we shall show that the sampling inequalities hold with certain probabilities for functions in $V_{N, \omega, \psi}^{p,q}(\Phi), V_{N,\psi}^{p,q}(\Phi, \mu, C_{K})$ and $V_{\psi}^{p,q}(\Phi, \delta, C_{K}).$  The result for  $V_{N, \omega, \psi}^{p,q}(\Phi)$ states as follows:

\begin{thm} \label{thm3.5}
Let $\Phi, \psi $ and $\rho $ satisfy assumptions $(A_{1}), (A_{2})$ and $(A_{3})$ respectively. Further, let $\{(x_{j}, y_{k})\}_{j,k \in \mathbb{N}}$ be a sequence of i.i.d. random variables that are drawn from a general probability distribution over $C_{K}= [-K_{1}, K_{1}]\times [-K_{2}, K_{2}]^{d}$ and whose density function is $\rho$. Then for any $ \gamma \in (0,1), \; 0< \omega \leq \|\psi\|_{L^{1,1}(C_{K})}$ and

\begin{eqnarray}
   && nm    > \left( \dfrac{54r \sqrt{2} (\ln 2)(2N+1)^{(d+1)}\|\psi\|_{L^{1,1}(C_{K})}}{\left(\dfrac{\gamma \mathcal{C}_{\rho, 1} \left(c^{*} \|\psi\|_{L^{1,1}(C_{K})}\right)^{(1-pq)} \omega^{pq}}{(2K_{1})^{(q-1)} (2K_{2})^{d(p-1)}} \right)^{2}} \right )   \nonumber \\
  &&\quad \quad \quad \times  \left( 2\left(\dfrac{\gamma \mathcal{C}_{\rho, 1} \left(c^{*} \|\psi\|_{L^{1,1}(C_{K})}\right)^{(1-pq)} \omega^{pq}}{(2K_{1})^{(q-1)} (2K_{2})^{d(p-1)}} \right) + 81 \|\psi\|_{L^{1,1}(C_{K})}\right),\nonumber \\
   \label{eqnm}
   \end{eqnarray}

\noindent
 the sampling inequality
 
  \begin{equation} 
  \mathcal{A}_{\gamma, \omega}\|f\|_{L^{p,q}(\mathbb{R}\times \mathbb{R}^{d})} \leq \bigg \|\{(f*\psi)(x_{j}, y_{k})\}_{\substack{j=1,2,\dots, n;\\ k=1,2,\dots, m}} \bigg \|_{l^{p,q}} \leq \mathcal{B}_{\gamma, \omega}\|f\|_{L^{p,q}(\mathbb{R}\times \mathbb{R}^{d})} \label{3eq1}
  \end{equation} 
   holds  with probability at least  $1 - \mathcal{A}_{1}  e^{(-nm \beta_{1} )}  - \mathcal{A}_{2}   e^{( -nm \beta_{2})}$ for every $f$ in $V_{N,\omega, \psi}^{p,q}(\Phi), $ 
  \noindent
where \begin{eqnarray*}
\mathcal{A}_{\gamma, \omega} &=& \dfrac{(1-\gamma) \mathcal{C}_{\rho, 1}(c^{*} \|\psi\|_{L^{1,1}(C_{K})})^{(1-pq)} \omega^{pq}}{(2K_{1})^{(q-1)}(2K_{2})^{d(p-1)}} n^{\frac{1}{p}} m^{\frac{1}{q}},\\
\mathcal{B}_{\gamma, \omega} &=& \dfrac{ \mathcal{C}_{\rho, 2}\|\psi\|_{L^{1,1}(C_{K})}}{(2K_{1})^{\left(\frac{1-p}{p} \right)}(2K_{2})^{d \left(\frac{1-q}{q} \right)}}nm + \dfrac{\gamma \mathcal{C}_{\rho, 1}(c^{*}\|\psi\|_{L^{1,1}(C_{K})})^{(1-pq)} \omega^{pq}}{(2K_{1})^{(q-1)}(2K_{2})^{d(p-1)}}nm,\\
\mathcal{A}_{1} &=& 2 \exp \bigg(r(2N+1)^{(d+1)} \ln\big(4c^{*} + 1\big) \bigg) ,\\
\beta_{1} &=& \dfrac{(2K_{1})^{(1-q)} (2K_{2})^{d(1-p)}\left(\frac{\sqrt{3}}{2}\gamma \mathcal{C}_{\rho, 1} \left(\frac{\omega}{c^{*}\|\psi\|_{L^{1,1}(C_{K})}}\right )^{pq}\right)^{2}}{6(2K_{1})^{(q-1)} (2K_{2})^{d(p-1)}+\gamma \mathcal{C}_{\rho, 1} \left(\frac{\omega}{c^{*}\|\psi\|_{L^{1,1}(C_{K})}}\right )^{pq}}, \\
\mathcal{A}_{2} &=& \dfrac{4 \bigg((2c^{*} +\frac{1}{4})(c^{*}+ \frac{1}{4}) \bigg )^{r(2N+1)^{(d+1)}}}{3r(\ln 2)^{2} (2N+1)^{(d+1)}},\\
\beta_{2} &=&  \dfrac{(2K_{1})^{(1-q)} (2K_{2})^{d(1-p)}\left(\gamma \mathcal{C}_{\rho, 1} \left(\frac{\omega}{c^{*}\|\psi\|_{L^{1,1}(C_{K})}}\right )^{pq} c^{*}\right)^{2}}{18 \sqrt{2}\left(81(2K_{1})^{(q-1)} (2K_{2})^{d(p-1)}+ 2 \gamma \mathcal{C}_{\rho, 1} \left(\frac{\omega}{c^{*}\|\psi\|_{L^{1,1}(C_{K})}}\right )^{pq}c^{*}\right)},   \ \text{and} \\
c^{*} &=& \dfrac{4 \tilde{c}}{2^{ \left(\frac{p+q}{pq} \right)} \alpha_{1}} \left(\sum_{k_{1} \in \mathbb{Z}} \dfrac{1}{\left(1 + |k_{1}| \right )^{ \left(\frac{s_{1}p}{p-1} \right)}}\right)^{ \left(\frac{p-1}{p} \right)} \left(\sum_{k_{2} \in \mathbb{Z}^{d}} \dfrac{1}{\left(1 + |k_{2}| \right)^{ \left(\frac{s_{2}q}{q-1} \right)}}\right)^{ \left(\frac{q-1}{q} \right)}.
\end{eqnarray*}

\end{thm}

\bigskip
\noindent
In order to prove Theorem \ref{thm3.5}, we consider for $f \in V^{p,q}(\Phi)$ and a sequence of i.i.d. random variables  $\{(x_{j}, y_{k})\}_{j,k \in \mathbb{N}}$  that are drawn from a general probability distribution over $C_{K}$ whose density function is $\rho,$  the random variables $ Y_{j,k}(f), j,k \in \mathbb{N}$, defined by \begin{equation}
Y_{j,k}(f) := |(f \ast \psi)(x_{j}, y_{k})| -   \int_{C_{K}}\rho(x,y)|(f \ast \psi)(x, y)| dx dy \label{3eq2}
\end{equation}
and a few of its properties, namely its expectation, variance and certain norm estimates. The proofs of these properties and the subsequent analysis require the following variants of Young's inequality.
\begin{lem}
Let $1 \leq p,q,r \leq \infty $ with $\frac{1}{p}+ \frac{1}{q} = \frac{1}{r} +1$. Suppose $f \in L^{p}(\mathbb{R}^{d}), g \in L^{q}(\mathbb{R}^{d})$ with $supp(g) \subset C_{K}$, where $C_{K}$ denotes the cube $[-K, K]^{d}.$ Then,
 \begin{equation}
\|f*g\|_{L^{r}(C_{K})} \leq \|f\|_{L^{p}(C_{2K})} \|g\|_{L^{q}(C_{K})}. \label{3eq3}
\end{equation}  \label{young1}
\end{lem}
\begin{proof}
Let $1 <p,q < \infty$  and $p',q'$ be their conjugate exponents respectively.  
Consider for $x \in C_{K},$
\begin{eqnarray*}
&&|(f*g)(x)|\\&& \quad  \leq \int_{\mathbb{R}^{d}}|f(x-y)||g(y)|dy \\
 && \quad = \int_{C_{K}}|f(x-y)||g(y)|dy  \\
 && \quad = \int_{C_{K}}|f(x-y)|^{\frac{p}{r}} |f(x-y)|^{(1- \frac{p}{r})}|g(y)|^{\frac{q}{r}} |g(y)|^{(1- \frac{q}{r})}dy \\
 && \quad= \int_{C_{K}} \big( |f(x-y)|^{p} |g(y)|^{q}   \big) ^{\frac{1}{r}} |f(x-y)|^{(1- \frac{p}{r})}|g(y)|^{(1- \frac{q}{r})} dy. \\ 
 \end{eqnarray*}
 As $\frac{1}{p'}+ \frac{1}{q'} + \frac{1}{r}=1,$ we get 
  \begin{eqnarray*}
  &&|(f*g)(x)|\\ && \quad \leq  \bigg ( \int_{C_{K}} \bigg|\big( |f(x-y)|^{p} |g(y)|^{q}   \big)^{\frac{1}{r}}\bigg| ^{r}dy \bigg )^{\frac{1}{r}} \bigg( \int_{C_{K}} \big ( |f(x-y)|^{(1- \frac{p}{r})}\big )^{q'}dy \bigg)^{\frac{1}{q'}} \\    
 && \quad \quad \quad \quad \quad \times \bigg( \int_{C_{K}}\big(|g(y)|^{(1- \frac{q}{r})} \big )^{p'} dy \bigg )^{\frac{1}{p'}} ,\\
 \end{eqnarray*}
 by Holder's inequality.
 Now using the relations  $q= p'(1- \frac{q}{r})$ and $p= q'(1- \frac{p}{r}),$ we obtain
 \begin{eqnarray*}
 &&|(f*g)(x)| \\
 &&  \leq \bigg( \int_{C_{K}}  |f(x-y)|^{p} |g(y)|^{q}   dy \bigg )^{\frac{1}{r}} \bigg( \int_{C_{K}} |f(x-y)|^{p}dy \bigg)^{\frac{1}{q'}}  \bigg( \int_{C_{K}} |g(y)|^{q}dy \bigg)^{\frac{1}{p'}}\\
 && \leq \bigg( \int_{C_{K}}  |f(x-y)|^{p} |g(y)|^{q}   dy \bigg )^{\frac{1}{r}}  \|f\|^{\frac{p}{q'}}_{L^{p}(C_{2K})} \|g\|^{\frac{q}{p'}}_{L^{q}(C_{K})}\\
 &&  \leq  \bigg( \int_{C_{K}}  |f(x-y)|^{p} |g(y)|^{q}   dy \bigg )^{\frac{1}{r}}  \|f\|^{(1-\frac{p}{r})}_{L^{p}(C_{2K})} \|g\|^{(1-\frac{q}{r})}_{L^{q}(C_{K})} ,
\end{eqnarray*} 
which implies
 \begin{eqnarray*}
&& \int_{C_{K}}|(f*g)(x)|^{r}dx \\
&& \leq \|f\|^{r-p}_{L^{p}(C_{2K})} \|g\|^{r-q}_{L^{q}(C_{K})} \int_{C_{K}}\bigg( \int_{C_{K}}  |f(x-y)|^{p} |g(y)|^{q}   dy \bigg )dx \\
&& = \|f\|^{r-p}_{L^{p}(C_{2K})} \|g\|^{r-q}_{L^{q}(C_{K})} \int_{C_{K}} |g(y)|^{q} \bigg( \int_{C_{K}}  |f(x-y)|^{p} dx   \bigg )dy \\
&& \leq \|f\|^{r}_{L^{p}(C_{2K})} \|g\|^{r-q}_{L^{q}(C_{K})} \int_{C_{K}} |g(y)|^{q} dy\\
&& =  \|f\|^{r}_{L^{p}(C_{2K})} \|g\|^{r}_{L^{q}(C_{K})},
\end{eqnarray*}  
thereby proving the desired inequality \eqref{3eq3}. The proof of \eqref{3eq3} for the other cases is obvious. 
\end{proof}

\noindent
The version of Young's inequality for mixed Lebesgue spaces is given in \cite{RuiBei}. The analogous result for the compact subset $C_{K} = [-K_{1}, K_{1}] \times [-K_{2}, K_{2}]^{d}$ is given below.

\begin{lem}
 Let $1<p,q< \infty$, $f \in L^{p,q}(\mathbb{R} \times \mathbb{R}^{d})$ and $g \in L^{1}(\mathbb{R}^{d+1})$ with supp$(g) \subset C_{K}$. Then,  we have
\begin{equation}
\|f*g\|_{L^{p,q}(C_{K})} \leq \|f\|_{L^{p,q}(C_{2K})}\|g\|_{L^{1,1}(C_{K})}.  \label{3eq4}
\end{equation}
Also \begin{equation}
\|f*g\|_{L^{\infty, \infty}(C_{K})} \leq \|f\|_{L^{\infty, \infty}(C_{2K})}\|g\|_{L^{1,1}(C_{K})}. \label{0eqn} 
\end{equation}
   \label{0lemma}
\end{lem}

\begin{proof}
For a fixed $x \in [-K_{1}, K_{1}],$ consider the functions $f_{x}, g_{x}$ on $[-K_{2}, K_{2}]^{d},$ given by $f_{x}(y)=f(x,y)$ and $g_{x}(y)= g(x,y), y \in [-K_{2}, K_{2}]^{d}.$ \\
\noindent
Then,
\begin{eqnarray*}
&& \|f*g\|^{p}_{L^{p,q}(C_{K})} \\
&& = \int_{[-K_{1}, K_{1}]} \bigg(\int_{[-K_{2}, K_{2}]^{d}}\bigg|\int_{\mathbb{R}} \int_{\mathbb{R}^{d}} f(x-x', y-y')g(x', y')dy' dx' \bigg|^{q} dy \bigg)^{\frac{p}{q}} dx \\
&& = \int_{[-K_{1}, K_{1}]} \bigg(\int_{[-K_{2}, K_{2}]^{d}}\bigg|\int_{\mathbb{R}} (f_{x-x'}*g_{x'})(y) dx' \bigg|^{q} dy \bigg)^{\frac{p}{q}} dx \\
&& = \int_{[-K_{1}, K_{1}]} \bigg(\int_{[-K_{2}, K_{2}]^{d}}\bigg|\int_{[-K_{1}, K_{1}]} (f_{x-x'}*g_{x'})(y) dx' \bigg|^{q} dy \bigg)^{\frac{p}{q}} dx\\
&& = \int_{[-K_{1}, K_{1}]} \bigg \|\int_{[-K_{1}, K_{1}]}(f_{x-x'}* g_{x'})(\cdot)dx' \bigg \|^{p}_{L^{q}([-K_{2}, K_{2}]^{d} )} dx.
\end{eqnarray*}  
\noindent
Applying the Minkowski's integral inequality and Lemma \ref{young1}, we have 
\begin{eqnarray*}&&  \bigg \|\int_{[-K_{1}, K_{1}]}(f_{x-x'}* g_{x'})(\cdot)dx' \bigg \|_{L^{q}([-K_{2}, K_{2}]^{d})} \\
&& \leq \int_{[-K_{1}, K_{1}]}\|(f_{x-x'}* g_{x'})(\cdot) \|_{L^{q}([-K_{2}, K_{2}]^{d})}dx'  \\
&& \leq \int_{[-K_{1}, K_{1}]}\|f_{x-x'}\|_{L^{q}([-2K_{2}, 2K_{2}]^{d})} \|g_{x'} \|_{L^{1}([-K_{2}, K_{2}]^{d})}dx' .
\end{eqnarray*} 
For $x \in [-K_{1}, K_{1}],$ let $\tilde{f}(x)= \|f_{x}\|_{L^{q}([-2K_{2},2 K_{2}]^{d})}$ and $\tilde{g}(x)= \|g_{x}\|_{L^{1}([K_{2}, K_{2}]^{d})}.$ Then supp$(\tilde{g}) \subset [-K_{1}, K_{1}]$ and we have

\begin{eqnarray*}
&& \|f*g\|^{p}_{L^{p,q}(C_{K})} \\
&& \leq  \int_{[-K_{1}, K_{1}]} \bigg (\int_{[-K_{1}, K_{1}]}\|f_{x-x'}\|_{L^{q}([-2K_{2}, 2K_{2}]^{d})} \|g_{x'} \|_{L^{1}([K_{2}, K_{2}]^{d})}dx' \bigg )^{p} dx \\
&& = \int_{[-K_{1}, K_{1}]} \bigg (\int_{[-K_{1}, K_{1}]}\tilde{f}(x-x') \tilde{g}(x') dx' \bigg )^{p} dx \\
&& = \int_{[-K_{1}, K_{1}]} |(\tilde{f}* \tilde{g})(x)|^{p} dx = \|\tilde{f}* \tilde{g}\|^{p}_{L^{p}([-K_{1}, K_{1}])} .
\end{eqnarray*}
Appealing to Lemma \ref{young1} once again, we get 
\begin{equation}
 \|f*g\|^{p}_{L^{p,q}(C_{K})} \leq \|\tilde{f}\|^{p}_{L^{p}([-2K_{1}, 2K_{1}])} \| \tilde{g}\|^{p}_{L^{1}([-K_{1}, K_{1}])}. \label{eqn*}
\end{equation}

\noindent
Further we observe that 
\begin{eqnarray*}
\|\tilde{f}\|^{p}_{L^{p}([-2K_{1}, 2K_{1}])} 
&=&\int_{[-2K_{1}, 2K_{1}]} \|f_{x}\|^{p}_{L^{q}([-2K_{2}, 2K_{2}]^{d})} dx\\ &=& \int_{[-2K_{1},2 K_{1}]} \bigg(\int_{[-2K_{2}, 2K_{2}]^{d}}|f(x,y)|^{q} dy \bigg )^{\frac{p}{q}} dx\\
&=& \|f\|^{p}_{L^{p,q}(C_{2K})}.
\end{eqnarray*} 
Similarly,
\begin{eqnarray*}
\|\tilde{g}\|^{p}_{L^{1}([-K_{1}, K_{1}])} 
&=& \bigg ( \int_{[-K_{1}, K_{1}]} \|g_{x}\|_{L^{1}([-K_{2}, K_{2}]^{d})} dx \bigg )^{p}\\ &=& \bigg( \int_{[-K_{1}, K_{1}]} \int_{[-K_{2}, K_{2}]^{d}}|g(x,y)| dy  dx \bigg )^{p}\\
&=& \|g\|^{p}_{L^{1,1}(C_{K})}.
\end{eqnarray*}
From \eqref{eqn*}, it then follows that$$\|f*g\|_{L^{p,q}(C_{K})} \leq \|f\|_{L^{p,q}(C_{2K})} \|g\|_{L^{1,1}(C_{K})}. $$
The proof of \eqref{0eqn} is obvious.
\end{proof}

\noindent
Now we shall look into the properties satisfied by the random variables ${Y_{j,k}(f)}$ defined in \eqref{3eq2}. 
\begin{flalign}
(1) & \ \text{The expectation} \  \mathbb{E}[Y_{j,k}(f)]= 0.  & \label{3eq5}
\end{flalign}
\begin{flalign*}
(2) & \ \|Y_{j,k}(f)\|_{l^{\infty, \infty}} \leq \|f\|_{L^{\infty, \infty}(C_{2K})} \|\psi\|_{L^{1,1}(C_{K})}.
&  \end{flalign*}
\noindent
$Proof$. As  $ \displaystyle \int_{C_{K}} \rho(x,y)|(f*\psi)(x,y)|dx dy \leq \|f*\psi\|_{L^{\infty, \infty}(C_{K})},$ we have \begin{eqnarray*}
&& \|Y_{j,k}(f)\|_{l^{\infty, \infty}} \\
 && \quad = \sup_{j,k}\bigg| \left|(f \ast \psi)(x_{j}, y_{k})\right| -  \int_{C_{K}} \rho(x,y)|(f*\psi)(x,y)|dx dy  \bigg|\\
 && \quad \leq \sup_{j,k} \max\bigg\{ \left|(f \ast \psi)(x_{j}, y_{k}) \right| ,  \int_{C_{K}} \rho(x,y)|(f*\psi)(x,y)|dx dy \bigg\}\\
&&  \quad \leq  \|f \ast \psi\|_{L^{\infty, \infty}(C_{K})}\\
&&  \quad \leq  \|f \|_{L^{\infty, \infty}(C_{2K})} \|\psi\|_{L^{1,1}(C_{K})},
\end{eqnarray*}
by \eqref{0eqn}.

\begin{flalign}
(3)  &\ \|Y_{j,k}(f)- Y_{j,k}(g)\|_{l^{\infty, \infty}} \leq 2 \|f-g\|_{L^{\infty, \infty}(C_{2K})} \|\psi\|_{L^{1,1}(C_{K})}. & \label{3eq7} 
 \end{flalign}
\noindent
$Proof$. Consider, 
\begin{eqnarray*}
&& \|Y_{j,k}(f)- Y_{j,k}(g)\|_{l^{\infty, \infty}}\\
&& \quad \leq  \sup_{j,k}\bigg| |((f-g)\ast \psi)(x_{j}, y_{k})| +   \int_{C_{K}} \rho(x,y)|((f-g)*\psi)(x,y) | dx dy \bigg|\\&& 
 \quad \leq  \sup_{(s,t) \in C_{K}}\bigg( | ((f-g) \ast \psi )(s, t)| + \|(f-g) \ast \psi\|_{L^{\infty, \infty}(C_{K})}  \int_{C_{K}} \rho(x,y)dx dy \bigg )\\
&& \quad \leq 2\|(f-g) \ast \psi\|_{L^{\infty, \infty}(C_{K})}\\
&& \quad \leq 2\|f -g\|_{L^{\infty, \infty}(C_{2K})} \|\psi\|_{L^{1,1}(C_{K})},
\end{eqnarray*}
by \eqref{0eqn}.

\begin{flalign*}
(4) &\ Var(Y_{j,k}(f)) \leq \|f\|^{2}_{L^{\infty, \infty}(C_{2K})} \|\psi\|^{2}_{L^{1,1}(C_{K})}.& 
\end{flalign*}
\noindent
$Proof$. It follows from \eqref{3eq5} that $Var(Y_{j,k}(f))= \mathbb{E}[Y_{j,k}(f)^{2}]. $
Now \begin{eqnarray*} \mathbb{E}[(Y_{j,k}(f))^{2}] &=& \mathbb{E}\left[|(f \ast \psi)(x_{j}, y_{k})|^{2} \right] + \mathbb{E}\left[\bigg( \int_{C_{K}} \rho(x,y)|(f*\psi)(x,y)|dx dy \bigg )^{2} \right ] \\
&& \quad  \quad - 2 \mathbb{E}\left[|(f \ast \psi)(x_{j}, y_{k})| \bigg ( \int_{C_{K}} \rho(x,y)|(f*\psi)(x,y)|dx dy \bigg ) \right ] \\
&  = & \mathbb{E}\left[|(f \ast \psi)(x_{j}, y_{k})|^{2} \right]- \bigg( \int_{C_{K}} \rho(x,y)|(f*\psi)(x,y)|dx dy \bigg)^{2},
\end{eqnarray*}
 which leads to the inequality 
\begin{eqnarray*}
Var(Y_{j,k}(f)) &\leq& \ \mathbb{E}[|(f \ast \psi)(x_{j}, y_{k})|^{2}]\\
&\leq& \|f \ast \psi\|^{2}_{L^{\infty, \infty}(C_{K})}\\
&\leq&   \|f\|^{2}_{L^{\infty, \infty}(C_{2K})} \|\psi\|^{2}_{L^{1,1}(C_{K})},
\end{eqnarray*} 
by \eqref{0eqn}.

\begin{flalign*}
(5) &\ Var \big(Y_{j,k}(f)-Y_{j,k}(g) \big ) \leq 4 \|f-g\|^{2}_{L^{\infty, \infty}(C_{2K})} \|\psi\|^{2}_{L^{1,1}(C_{K})} .& 
\end{flalign*}
\noindent
$Proof$.
 \begin{eqnarray*} && Var \big(Y_{j,k}(f)-Y_{j,k}(g) \big ) \\
 && \quad  = \mathbb{E}\big[\big(Y_{j,k}(f)-Y_{j,k}(g) \big )^{2}\big] \\
 && \quad \leq \mathbb{E}\left [\bigg(\left|\big( (f-g) \ast \psi \big )(x_{j}, y_{k}) \right| +  \int_{C_{K}}\rho(x,y) \left|\big((f-g)*\psi \big)(x,y)\right | dx dy \bigg)^{2} \right] \\
  && \quad = \mathbb{E}\big[|\big( (f-g) \ast \psi \big )(x_{j}, y_{k})|^{2}\big] +\bigg(  \int_{C_{K}}\rho(x,y) |\big((f-g)*\psi \big)(x,y)| dx dy\bigg)^{2}\\
 && \quad \quad \quad  2 \bigg(  \int_{C_{K}}\rho(x,y) |\big((f-g)*\psi \big)(x,y)| dx dy\bigg)\mathbb{E}\big[|\big( (f-g) \ast \psi \big )(x_{j}, y_{k})|\big]\\
&&  \quad \leq 4\|(f-g) \ast \psi\|^{2}_{L^{\infty, \infty}(C_{K})}\\
&& \quad \leq 4  \|f-g\|^{2}_{L^{\infty, \infty}(C_{2K})} \|\psi\|^{2}_{L^{1,1}(C_{K})},
\end{eqnarray*}
by \eqref{0eqn}.

\bigskip
\noindent
To proceed further, we shall consider the following lemmas. The relation between the $L^{\infty, \infty}$ and $L^{p,q}$ norms of functions in $V_{N}^{p,q}(\Phi)$ is given below.

\begin{lem} \cite{jiangli3}
Suppose that $\Phi$ satisfies \eqref{2eq1} and \eqref{2eq2}. Then for every function $f \in V_{N}^{p,q}(\Phi),$ we have 
$$\|f\|_{L^{\infty, \infty} (\mathbb{R}\times \mathbb{R}^{d})} \leq c'\|f\|_{L^{p,q}(\mathbb{R}\times \mathbb{R}^{d})},$$ 
where 
$\displaystyle c' = \dfrac{2^{\left( \frac{1}{p'} + \frac{1}{q'} \right)} \tilde{c}}{\alpha_{1}} \left(\sum_{k_{1} \in \mathbb{Z}} \dfrac{1}{\left( 1 + |k_{1}| \right )^{ s_{1} p'}}\right)^{ \frac{1}{p'} } \left(\sum_{k_{2} \in \mathbb{Z}^{d}} \dfrac{1}{ \left( 1 + |k_{2}| \right )^{ s_{2} q'}}\right)^{ \frac{1}{q'} } $
 and $p',q'$ denote the conjugate exponents of $p$ and $q$ respectively.\label{lem1}
\end{lem}

\noindent
As the sampling inequalities are given in a probabilistic sense, it is vital to consider the notion of covering numbers which help in the estimation of the probability.\\
For a compact set $C$ in a metric space, its covering number $\mathcal{N}(C, \epsilon), \epsilon > 0$ is the least number of balls of radius $\epsilon $ needed to cover $C.$ The following lemma gives an upper bound for the covering number of $V_{N}^{p,q, *}(\Phi).$

\begin{lem} \cite{jiangli3} Suppose that $\Phi$ satisfies \eqref{2eq1} and \eqref{2eq2}. Then the covering number of $V_{N}^{p,q, *}(\Phi)$ with respect to
 $\| \cdot \|_{L^{\infty, \infty}(\mathbb{R}\times \mathbb{R}^{d})}$
  is bounded by
   $$\mathcal{N}(V_{N}^{p,q,*}(\Phi), \epsilon) \leq \exp \left(r(2N+1)^{(d+1)} \ln\left(\frac{2c'}{\epsilon} + 1\right) \right),$$
    where $c'$ as in Lemma  \ref{lem1}. \label{lem2}
\end{lem}

\noindent
The analogue of the Bernstein's inequality  for multivariate random variables is as follows (see \cite{bennett, jiangli3}).

\begin{lem}
Let $Z_{j,k}(j= 1,2, \dots, n; k=1,2,\dots, m)$ be independent random variables with expected values $\mathbb{E}[Z_{j,k}]=0$ for all $j=1,2, \dots, n$ and $k=1,2, \dots, m.$ Assume that $Var(Z_{j,k}) \leq \sigma^{2}$  and $|Z_{j,k}|\leq M$ almost surely for all $j=1,2, \dots, n$ and $k=1,2, \dots, m.$ Then for any $\lambda \geq 0$, 
$$Prob\left( \bigg|\sum_{j=1}^{n} \sum_{k=1}^{m} Z_{j,k}\bigg| \geq \lambda\right) \leq 2 \exp \left( - \dfrac{\lambda^{2}}{2nm \sigma^{2} + \frac{2}{3} M \lambda} \right).$$ \label{lem3.2}
\end{lem}


\noindent
We shall now state and prove a crucial lemma for our analysis.
\begin{lem} Let $\Phi, \psi$ and $\rho$ be as in the hypothesis of Theorem \ref{thm3.5} and $Y_{j,k}$ be as defined in \eqref{3eq2}. Then for any $m, n, N \in \mathbb{N}$ and \\ 
$\lambda >  54 r \sqrt{2} (\ln 2) (2N+1)^{(d+1)} \bigg(1+ \bigg(1+\dfrac{3nm}{2 r \sqrt{2} (\ln 2) (2N+1)^{(d+1)}} \bigg)^{\frac{1}{2}} \bigg )\|\psi\|_{L^{1,1}(C_{K})},$
 the following inequality holds:

 \begin{eqnarray}
 && Prob\left(\sup_{f \in V_{N}^{p,q,*}(\Phi)} \bigg|\sum_{j=1}^{n} \sum_{k=1}^{m} Y_{j,k}(f)\bigg|
  \geq \lambda \right) \nonumber \\  
    && \quad \leq \mathcal{A}_{1} \exp \left( - \dfrac{\lambda^{2}}{4c^{*}\|\psi\|_{L^{1,1}(C_{K})}(2 n m c^{*}\|\psi\|_{L^{1,1}(C_{K})}+  \frac{\lambda}{3})}\right) \label{3eq10} \\ 
&&      \quad \quad + \mathcal{A}_{2} \exp \left( - \dfrac{\lambda^{2}}{18 \sqrt{2}\|\psi\|_{L^{1,1}(C_{K})}(81 n m \|\psi\|_{L^{1,1}(C_{K})}+  2\lambda)}\right), \nonumber \end{eqnarray}
where $$\mathcal{A}_{1}= 2 \exp \left(r(2N+1)^{(d+1)} \ln\big(4c^{*} + 1\big) \right),$$  $$\mathcal{A}_{2}= \dfrac{4 \left((2c^{*} +\frac{1}{4})(c^{*}+ \frac{1}{4}) \right )^{r(2N+1)^{(d+1)}}}{3r(\ln 2)^{2} (2N+1)^{(d+1)}}$$ and $c^{*}$ is as in Theorem \ref{thm3.5}.
  \label{lemm3.3}
\end{lem}

 The following theorem provides a sampling inequality for another subset $V^{p,q}_{N,\psi}(\Phi,\mu,C_K)$ of $L^{p,q}(\mathbb{R}\times\mathbb{R}^d)$, defined as in \eqref{rte1}.
\begin{thm}
Let $\Phi,\,\psi,\,\rho$ and the i.i.d random variables $\{(x_j,y_k)\}_{j,k\in\mathbb{N}}$ over $C_K$ be as in the hypothesis of Theorem \ref{thm3.5}. For $N\in\mathbb{N},\,0<\mu\leq 1$ and $0<\eta<\mu \mathcal{C}_{\rho,1}$, let $m,n\in\mathbb{N}$ be such that 
\begin{equation}
	nm>\dfrac{54r\sqrt{2}(\ln 2)(2N+1)^{(d+1)}}{\eta}\left(2+\frac{81}{\eta}\right).  \label{rteq3}
\end{equation}
 Then, the sampling inequality for the functions $f\in V^{p,q}_{N,\psi}(\Phi,\mu,C_K)$, namely,
\begin{align}
& \bigg(nm\|\psi\|_{L^{1,1}(C_K)}(\mu \mathcal{C}_{\rho,1}-\eta) \bigg )\|f\|_{L^{p,q}(\mathbb{R}\times\mathbb{R}^d)}\nonumber\\
&\qquad\leq\sum_{j=1}^{n}\sum_{k=1}^{m}|(f\ast\psi)(x_j,y_k)|\nonumber\\
&\qquad\qquad\leq \bigg(nm\|\psi\|_{L^{1,1}(C_K)}\left(\mathcal{C}_{\rho,2}(2K_1)^{(\frac{p-1}{p})}(2K_2)^{\frac{d(q-1)}{q}}+\eta \right) \bigg )\|f\|_{L^{p,q}(\mathbb{R}\times\mathbb{R}^d)},\label{E:Smplg_Ineq2}
\end{align}
holds with probability atleast \[1-\mathcal{A}_1 \exp\left(-nm\frac{3\eta^2}{4c^{*}(6c^{\ast}+\eta)}\right)-\mathcal{A}_2 \exp\left(-nm\frac{\eta^2}{18\sqrt{2}(81+2\eta)}\right),\]
where $\mathcal{A}_1, \mathcal{A}_2$ and $c^{\ast}$ are as in Lemma \ref{lemm3.3}. \label{rtthm}
\end{thm}

 \noindent
We shall now consider the sampling inequality for the set $V_{\psi}^{p,q}(\Phi, \delta, C_{K})$. This requires the following finite dimensional approximation.
 \begin{lem} \cite{jiangli3} \label{lemma2.1}
 Let $1 < p,q < \infty, \frac{1}{p} + \frac{1}{p'}=1$ and $ \frac{1}{q} + \frac{1}{q'}=1.$ Let  $f \in V^{p,q}(\Phi)$ and $\|f\|_{L^{p,q}(\mathbb{R} \times \mathbb{R}^{d})}=1,$ wherein $\Phi$ satisfies the assumption $(A_{1}).$ Also let $C_{K} = [-K_{1}, K_{1}] \times [-K_{2}, K_{2}]^{d}$ and $s=\min\{s_{1}, s_{2}\} + \frac{1}{p}+ \frac{d}{q} -(d+1).$  Then for any $\epsilon_{1},\epsilon_{2}>0,$ there exists  $f_{N} \in V_{N}^{p,q}(\Phi)$ such that
 \begin{equation*}
 \|f - f_{N}\|_{L^{p,q}(C_{K})} \leq \epsilon_{1} 
 \end{equation*}  
\begin{eqnarray*} && \text{if}\ N \geq N_{1}(K_{1}, K_{2}, \epsilon_{1}) =  \max \{K_{1}, K_{2}\} + \\ && \quad \quad \quad \quad \quad \quad \quad  \bigg(\dfrac{\tilde{c}(K_{1})^{\frac{1}{p}}(K_{2})^{\frac{d}{q}} d^{\frac{1}{q'}}(1+K_{2})^{\left(\frac{d-1}{q'}+\frac{1}{p'} \right)}2^{(d+1)}}{\alpha_{1} \epsilon_{1}(s_{2}q'-d)^{\frac{1}{q'}}} \\
&& \quad \quad \quad \quad \quad \quad \quad \quad \quad + \dfrac{\tilde{c}(K_{1})^{\frac{1}{p}}(K_{2})^{\frac{d}{q}} (1+K_{1})^{\frac{d}{q'}}2^{(d+1)}}{\alpha_{1} \epsilon_{1}(s_{1}p'-1)^{\frac{1}{p'}}}   \\  
&& \quad \quad \quad \quad \quad \quad \quad \quad \quad \quad \quad + \dfrac{\tilde{c}(K_{1})^{\frac{1}{p}}(K_{2})^{\frac{d}{q}} d^{\frac{1}{q'}} (1+K_{2})^{\left(\frac{d-1}{q'} \right)}2^{(d+1)}}{\alpha_{1} \epsilon_{1}(s_{1}p'-1)^{\frac{1}{p'}}(s_{2}q'-d)^{\frac{1}{q'}}}  \bigg)^{\frac{1}{s}}
\end{eqnarray*} 
and
 \begin{equation*}
\|f-f_{N}\|_{L^{\infty, \infty}(C_{K})} \leq \epsilon_{2}
\end{equation*}
 
  \begin{eqnarray*}
  && \text{if} \ N \geq N_{2}(K_{1}, K_{2}, \epsilon_{2}) =  \max \{K_{1}, K_{2}\} + \\ && \quad \quad \quad \quad \quad \quad \quad \bigg(\dfrac{\tilde{c}d^{\frac{1}{q'}}(1+K_{2})^{\left(\frac{d-1}{q'}+\frac{1}{p'} \right)}2^{\left(\frac{1}{p'} + \frac{d}{q'} \right)}}{\alpha_{1} \epsilon_{2}(s_{2}q'-d)^{\frac{1}{q'}}} + \dfrac{\tilde{c} (1+K_{1})^{\frac{d}{q'}}2^{ \left(\frac{1}{p'} + \frac{d}{q'} \right) }}{\alpha_{1} \epsilon_{2}(s_{1}p'-1)^{\frac{1}{p'}}}   \\  
&& \quad \quad \quad \quad \quad \quad \quad \quad \quad \quad \quad + \dfrac{\tilde{c} d^{\frac{1}{q'}} (1+K_{2})^{\left(\frac{d-1}{q'} \right)}2^{ \left(\frac{1}{p'} + \frac{d}{q'} \right )}}{\alpha_{1} \epsilon_{2}(s_{1}p'-1)^{\frac{1}{p'}}(s_{2}q'-d)^{\frac{1}{q'}}}  \bigg)^{\frac{1}{s}}.
  \end{eqnarray*}
 \end{lem}

\begin{thm}
Let $\Phi, \psi, \{(x_{j}, y_{k})\}_{j,k \in \mathbb{N}},$ cuboid $C_{K}=[-K_{1}, K_{1}]\times [-K_{2}, K_{2}]^{d}$ and the probability density function $\rho$ be as in the hypothesis of Theorem \eqref{thm3.5}. Then for $0 < \delta < 1$, $ 0 < \epsilon < 1- \delta $ and $0 < \gamma < 1- \dfrac{\epsilon}{(1-\delta-\epsilon)^{(1+pq)}},$    the sampling inequality 
\begin{equation}
A\|f\|_{L^{p,q}(\mathbb{R} \times \mathbb{R}^{d})} \leq \bigg\|\{(f*\psi)(x_{j}, y_{k})\}_{\substack{j=1,2,\dots, n;\\ k=1,2,\dots, m}}\bigg\|_{l^{p,q}} \leq B \|f\|_{L^{p,q}(\mathbb{R} \times \mathbb{R}^{d})} \label{3eq24}
\end{equation}  

\noindent
holds uniformly for all 
$f \in V_{\psi}^{p,q}(\Phi, \delta, C_{K})$
 with probability at least
 $$ 1 - \mathcal{A}_{1} e^{ (-nm \beta_{1} )} - \mathcal{A}_{2} e^{(-nm \beta_{2} )},$$ provided $n,m$ satisfies \eqref{eqnm}, where
\begin{eqnarray*}
&& A= \dfrac{\mathcal{C}_{\rho, 1} (c^{*})^{(1-pq)} \|\psi\|_{L^{1,1}(C_{K})}\left(  (1- \gamma)(1-\delta-\epsilon)^{(1+pq)} -\epsilon \right )}{(2K_{1})^{(q-1)} (2K_{2})^{d(p-1)}}n^{\frac{1}{p}} m^{\frac{1}{q}}, \\
&& B= \dfrac{\alpha_{2}  \|\psi\|_{L^{1,1}(C_{K})}}{\alpha_{1}} \left(\dfrac{\mathcal{C}_{\rho, 2}}{(2K_{1})^{\left(\frac{1-p}{p} \right)} (2K_{2})^{d \left(\frac{1-q}{q} \right)}} +  \dfrac{\gamma \mathcal{C}_{\rho, 1} (c^{*})^{(1-pq)} (1-\delta-\epsilon)^{pq}}{(2K_{1})^{(q-1)} (2K_{2})^{d(p-1)}}\right) nm\\ && \quad \quad  + \ \dfrac{  \epsilon \mathcal{C}_{\rho, 1} (c^{*})^{(1-pq)} \|\psi\|_{L^{1,1}(C_{K})}}{(2K_{1})^{(q-1)}(2K_{2})^{d(p-1)}} n^{\frac{1}{p}} m^{\frac{1}{q}}, 
\end{eqnarray*}

\noindent
the constants $c^{*}, \mathcal{A}_{1}, \mathcal{A}_{2}, \beta_{1}$ and $\beta_{2}$ are as in Theorem \ref{thm3.5} with  $\omega= (1-\delta - \epsilon)\|\psi\|_{L^{1,1}(C_{K})}.$ The constant $N$ appearing in $\mathcal{A}_{1}$ and $\mathcal{A}_{2}$ is given by  
 $$N=\max \left \{N_{1}\left(2K_{1}, 2K_{2}, \epsilon \right), N_{2}\left(2K_{1}, 2K_{2}, \dfrac{\epsilon \mathcal{C}_{\rho, 1} (c^{*})^{(1-pq)}}{(2K_{1})^{(q-1)} (2K_{2})^{d(p-1)}} \right) \right \},$$
  where $N_{1}$ and $N_{2}$ are as in Lemma \ref{lemma2.1}.
 \label{thm3.9}
\end{thm}

 \begin{rem}
We observe that the probabilities with which the sampling inequalities proved in Theorems \ref{thm3.5}, \ref{rtthm} and \ref{thm3.9} hold approach one when the sample size tends to infinity.
 \end{rem}

\section{Reconstruction using random average samples} \label{sec4}

In this section, we give reconstruction formulae for functions in the signal classes $V_{N, \omega, \psi}^{p,q}(\Phi)$ and $V_{N, \psi}^{p,q}(\Phi, \mu, C_{K}).$


\begin{thm} Let  $\Phi$ and $\psi$ satisfy the assumptions $(A_{1})$ and $(A_{2})$ respectively. Suppose $\{(x_{j}, y_{k}) \}_{j,k \in \mathbb{N}}$ is a sequence of i.i.d. random variables that are drawn from a general probability distribution over the cuboid $C_{K}=[-K_{1}, K_{1}] \times [-K_{2}, K_{2}]^{d}$ with the density function $\rho$ satisfying the assumption $(A_{3})$. If
\begin{equation}
\left\|\sum_{|k_{1}| \leq N}\sum_{|k_{2}| \leq N} \textbf{c}^{T}(k_{1}, k_{2}) (\Phi * \psi)(\cdot -k_{1}, \cdot -k_{2})\right\|_{L^{p,q}(C_{K})} \geq \tilde{\beta}\|\textbf{c}\|_{l^{p,q}}   \label{4eq1}
\end{equation} holds for all $\textbf{c} \in \left(l^{p,q}\big([-N, N] \times [-N, N]^{d}\big)\right)^{r}, N \in \mathbb{N}$ and  for some positive constant $\tilde{\beta}$, then for any $\gamma \in (0,1)$, there exists a finite sequence of functions $\bigg\{G_{j,k}:{(j,k) \in \{1,2, \dots , n\} \times \{1,2, \dots, m\}} \bigg\} $ such that
$$f(x, y)= \sum_{j=1}^{n} \sum_{k=1}^{m}(f*\psi)(x_{j}, y_{k}) G_{j,k}(x,y)$$ holds for all $f \in V_{N}^{p,q}(\Phi)$ with probability at least
 
\begin{eqnarray}
 && 1 - \mathcal{A}_{1}\exp\left ( -nm\dfrac{ (2K_{1})^{(1-q)}(2K_{2})^{d(1-p)}\left(\frac{\sqrt{3}}{2}\gamma \mathcal{C}_{\rho, 1} \left(\frac{\tilde{\beta}(\alpha_{2} c^{*})^{-1}}{\|\psi\|_{L^{1,1}(C_{K})}}\right )^{pq}\right)^{2}}{6(2K_{1})^{(q-1)}(2K_{2})^{d(p-1)}+\gamma \mathcal{C}_{\rho, 1} \left(\frac{\tilde{\beta}(\alpha_{2} c^{*})^{-1}}{\|\psi\|_{L^{1,1}(C_{K})}}\right )^{pq}} \right )\nonumber \\ 
 &&  - \mathcal{A}_{2} \exp \left ( -nm\dfrac{(2K_{1})^{(1-q)}(2K_{2})^{d(1-p)}\left(\gamma \mathcal{C}_{\rho, 1} \left(\frac{\tilde{\beta}(\alpha_{2} c^{*})^{-1}}{\|\psi\|_{L^{1,1}(C_{K})}}\right )^{pq} c^{*}\right)^{2}}{18 \sqrt{2}\left(81(2K_{1})^{(q-1)}(2K_{2})^{d(p-1)}+ 2 \gamma \mathcal{C}_{\rho, 1} \left(\frac{\tilde{\beta}(\alpha_{2} c^{*})^{-1}}{\|\psi\|_{L^{1,1}(C_{K})}}\right )^{pq}c^{*}\right)} \right ),\nonumber\\ \label{eqn41.1} \label{neweq}
  \end{eqnarray}

\noindent
 where $c^{*}, \mathcal{A}_{1}$ and $\mathcal{A}_{2}$ are positive constants as in Theorem \ref{thm3.5}. \label{thmrec1}
\end{thm}

\noindent
Using the sampling inequality provided in Theorem \ref{rtthm}, we shall give a reconstruction formula for the functions in  the set $V^{p,q}_{N,\psi}(\Phi,\mu,C_K),$ given by \eqref{rte1}.

\begin{thm}
Let $\Phi,\,\psi$ and $\rho$ satisfy the assumptions $(A_1)$, $(A_2)$ and $(A_3)$ respectively. Also, let $\{(x_j,y_k)\}_{j,k\in\mathbb{N}}$ denote a sequence of i.i.d random variables over a cuboid  $C_K\subset\mathbb{R}\times\mathbb{R}^d$ drawn from a probability distribution with probability density function $\rho$. For $N\in\mathbb{N},\,0<\mu\leq 1$ and $0<\eta<\mu \mathcal{C}_{\rho,1}$, let $m,n\in\mathbb{N}$ be such that  \eqref{rteq3} is satisfied.
Then, there exist functions $\{G_{j,k}\}_{\substack{j=1,2,\cdots ,n;\\k=1,2,\cdots ,m}}$ such that
every $f\in V^{p,q}_{N,\psi}(\Phi,\mu,C_K)$ can be reconstructed by 
\[f(x,y)=\sum_{j=1}^n\sum_{k=1}^m(f\ast\psi)(x_j,y_k)G_{j,k}(x,y),\quad (x,y)\in\mathbb{R}\times\mathbb{R}^d\]
with probability atleast \[1-\mathcal{A}_1 \exp\left(-nm\frac{3\eta^2}{4c^{\ast}(6c^{\ast}+\eta)}\right)-\mathcal{A}_2 \exp\left(-nm\frac{\eta^2}{18\sqrt{2}(81+2\eta)}\right),\]
where $\mathcal{A}_1, \mathcal{A}_2$ and $c^{\ast}$ are as in Lemma \ref{lemm3.3}. \label{rtthm2}
\end{thm}

\section{Examples and numerical simulation} \label{sec5} In this section, we give  some examples of generators $\phi_{i}$   which satisfy the assumption $(A_{1})$. We also validate the results obtained in the previous section numerically using some of these examples.

For $n \in {\mathbb{N}}$, we consider the cardinal B-spline of degree $n$ defined on ${\mathbb{R}}$ by  $$B_{n}(x)= (B_{0} * B_{0}* \cdots * B_{0})(x), (n \ \text{convolutions}),$$ where   $B_{0}(x)= \chi_{[-\frac{1}{2}, \frac{1}{2}]}(x) .$

\begin{ex}
We define $\phi:{\mathbb{R}}^2\rightarrow {\mathbb{R}}$ by  $\phi(x,y):=B_{n}(x)B_{n}(y), n \in \mathbb{N}.$ \label{example1}
\end{ex}

\begin{ex}For $r, N\in {\mathbb{N}},$ let $u_{i},v_{i}\in {\mathbb{Z}}$  be such that
$|u_{i}-u_{j}|>2 N$  and  $|v_{i}-v_{j}|>2 N$ for $1\le i,j\le r,i\neq j.$ Let  $\phi_{i}(x,y)= B_{n}(x - u_{i})B_{n}(y - v_{i}), n \in \mathbb{N}$ for $i= 1,2,\dots, r.$ \label{example2}
\end{ex}

\begin{ex}
 Let $\phi_{1}$ and $\phi_{2}$ be two compactly supported continuous functions  such that $\phi_1(\cdot-k_1,\cdot-k_2)$ is orthogonal to  $\phi_2(\cdot-k'_1,\cdot-k'_2)$ for every $k_1,k'_1,k_2,k'_2\in {\mathbb{Z}}$ and there exist positive constants $A_{1}, B_{1}, A_{2}$ and $B_{2}$ such that
$A_{1} \leq P_{\phi_{1}}(\omega) \leq B_{1} \; \text{a.e.}$ and
$A_{2} \leq P_{\phi_{2}}(\omega) \leq B_{2} \; \text{a.e.}.$
\label{example3}
\end{ex}

\noindent
For numerical implementations,  we consider the cuboid $C_{K} = [-2.5,2.5] \times [-2.5,2.5].$ The function $f(x,y)$ and the averaging function $\psi(x,y)$ are as defined below:
\begin{eqnarray*}
f(x,y) &=& 3B_{2}(x)B_{2}(y-1) - 5B_{2}(x+1)B_{2}(y) \ \ \text{and} \\
\psi(x,y) &=&  \chi_{[-\frac{1}{8}, \frac{1}{8}] \times [-\frac{1}{8}, \frac{1}{8}]}(x,y).
\end{eqnarray*}
One can easily verify that the functions $f$ and $\psi$ satisfy the conditions of Theorem \ref{thmrec1}. Using the reconstruction formula provided in Theorem \ref{thmrec1}, the simulation is performed for various values of the sample size $nm.$  The graphical representations of the function $f$ and its reconstructed version  $\widetilde{f}$ corresponding to 25 (m=5 and n=5) random samples are shown in Figures \ref{figure1a} and \ref{figure1b} respectively.

\begin{figure}[H]
\includegraphics[scale=0.5]{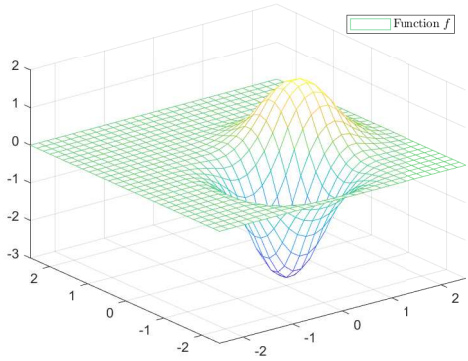}
\caption{The 3D plot of the function $f(x,y).$}
\label{figure1a}
\end{figure}
\begin{figure}[H]
\includegraphics[scale=0.5]{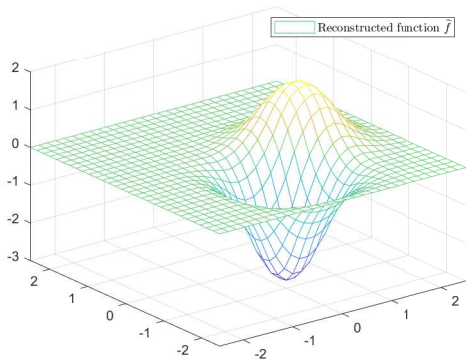}
\caption{The 3D plot of  $\tilde{f}$ corresponding to 25 samples (n=5, m=5).}
\label{figure1b}
\end{figure}

The reconstruction error $\|f-\widetilde{f}\|$ is  also computed with respect to $L^{\infty}, L^{1}$ and $L^{2}$ norms for various sample sizes. The numerical results are presented in Table \ref{table1}.
\begin{table}
\begin{tabular}{| *{5}{c|} }
    \hline
\multicolumn{2}{|c|}{Sample size}    & \multicolumn{3}{|c|}{Reconstruction Error}
 \\
     \hline
$n$   &   $m$  &   $\|f-\widetilde{f}\|_{L^{\infty}(C_{K})}$  &  $\|f-\widetilde{f}\|_{L^{1}(C_{K})}$  &   $\|f-\widetilde{f}\|_{L^{2}(C_{K})}$    \\
    \hline
$5$   &   $5$  &   $2.3315 \ e^{-15}$  &   $7.6548 \ e^{-15}$ &   $7.7786 \ e^{-30}$ \\
    \hline
$7$   &   $7$         &  $2.2204 \ e^{-15}$      &             $3.8640  \ e^{-15}$&  $2.5115 \ e^{-30}$ \\
    \hline
    $10$   &   $10$    &  $1.3323 \ e^{-15}$     &  $3.1535 \ e^{-15}$     &  $1.5283 \ e^{-30}$             \\
    \hline
\end{tabular}
\caption {The reconstruction error $\|f- \widetilde{f}\| $.}
\label{table1}
\end{table}
\vskip 1em
\noindent
 Further, to test Theorem \ref{rtthm2} numerically, we  consider  $C_{K} = [-3,3]\times [-3,3],$
\begin{eqnarray*}
f(x,y) &=& B_{1}(x)B_{1}(y) +3B_{1}(x-1)B_{1}(y-1) \ \ \text{and} \\
\psi(x,y) &=& \chi_{[\frac{1}{2}, \frac{3}{2}] \times [\frac{1}{2}, \frac{3}{2}]}.
\end{eqnarray*}
It can be easily shown that $f \in V_{1,\psi}^{2,2}(\Phi, \mu, C_{K}),$ where $ 0 < \mu \leq 1.$ The reconstruction formula  in Theorem \ref{rtthm2} has been used for reconstructing the function $f$ for different sample sizes.  The  Figures \ref{figure2a} and \ref{figure2b}  show the function $f$ and its reconstructed version $\widetilde{f}$ corresponding to 25 random samples($m=5, n=5$). The errors in $L^{\infty}, L^{1}$ and $L^{2}$ norms are calculated and the numerical values are given in Table \ref{table2}.

\begin{figure}[H]
\includegraphics[scale=0.5]{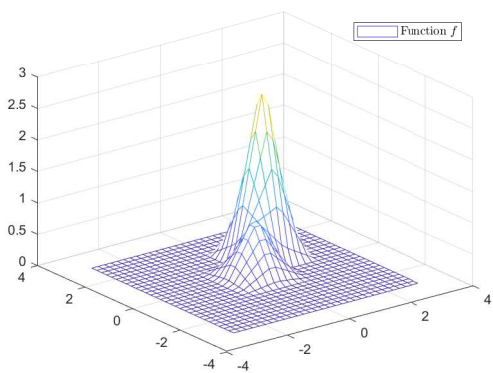}

\caption{The 3D plot of the function $f(x,y)$}
\label{figure2a}
\end{figure}

\begin{figure}[H]
\includegraphics[scale=0.5]{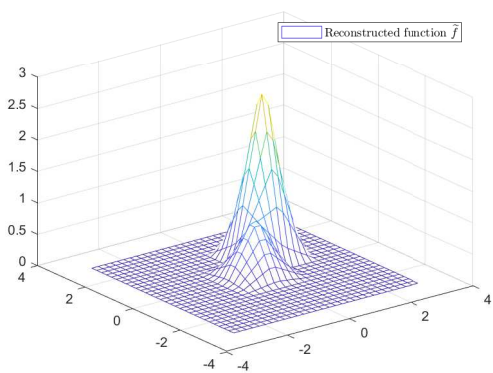}
\caption{The 3D plot of $\widetilde{f}$ for $n=5, m=5.$} \label{figure2b}
\end{figure}

\begin{table}[h]
\begin{tabular}{| *{5}{c|} }
    \hline
\multicolumn{2}{|c|}{Sample size}    & \multicolumn{3}{|c|}{Reconstruction Error}
 \\
     \hline
$n$  &  $m$  &  $\|f-\widetilde{f}\|_{L^{\infty}(C_{K})}$  &  $\|f-\widetilde{f}\|_{L^{1}(C_{K})}$  &   $\|f-\widetilde{f}\|_{L^{2}(C_{K})}$    \\
    \hline
$5$   &   $5$  &   $7.8930 \ e^{-14}$  &   $1.6052 \ e^{-13}$ &   $4.7277 \ e^{-27}$ \\
    \hline
$7$   &   $7$         &  $5.5227 \ e^{-14}$      &             $5.5811  \ e^{-14}$&  $1.8715 \ e^{-27}$ \\
    \hline
    $10$   &   $10$    &  $8.8818 \ e^{-16}$     &  $2.3891 \ e^{-15}$     &  $7.3617 \ e^{-31}$             \\
    \hline 
\end{tabular} 
\caption{The reconstruction error $\|f- \widetilde{f}\|$.}
\label{table2}
\end{table}

\vskip 1 em
\noindent {\bf Acknowledgment:}
The second author S. Arati acknowledges the financial support of National Board for Higher Mathematics, Department of Atomic Energy(Government of India). The third author P. Devaraj acknowledges the financial support of the Department of Science and Technology(Government of India) under the research grant DST-SERB Research Grant MTR/2018/000559.

\end{document}